\documentclass[12pt]{amsart}
\title{A surprising formula for Sobolev norms}% and related topics}

\author{Ha\"\i m Brezis}
\address{Department of Mathematics\\
  Rutgers University, Hill Center, Busch Campus\\ 
  110 Frelinghuysen Road, Piscataway, NJ 08854, USA}
\address{Departments of Mathematics and Computer Science\\ Technion, Israel Institute of Technology\\ 32.000 Haifa, Israel}
\address{Laboratoire Jacques-Louis Lions\\
  Sorbonne Universit\'es, UPMC Universit\'e Paris-6, 4  place Jussieu\\
  75005 Paris, France}
\email{brezis@math.rutgers.edu}
  
\author{Jean Van Schaftingen}
\address{Universit\'e catholique de Louvain\\ 
Institut de Recherche en Math\'ematique et Physique\\
Chemin du Cyclotron 2 bte L7.01.01\\
1348 Louvain-la-Neuve\\
Belgium}
\email{Jean.VanSchaftingen@uclouvain.be}

\author{Po-Lam Yung}
\address{Mathematical Sciences Institute,
Australian National University,
125 Science Road,
Canberra ACT 2601,
Australia}
\address{Department of Mathematics,
The Chinese University of Hong Kong,
Ma Liu Shui,
Hong Kong}
\email{PoLam.Yung@anu.edu.au}
\email{plyung@math.cuhk.edu.hk}

\keywords{Fractional Sobolev space  $|$  Marcinkiewicz space  $|$  fractional Gagliardo--Nirenberg interpolation inequalities} 

\usepackage{hypbmsec,fullpage}

\newcommand{\defeq}{:=}

\newcommand{\abs}[1]{{\lvert #1 \rvert}}

\newcommand{\norm}[2][]{{\lVert #2 \rVert}_{#1}}
\newcommand{\seminorm}[2][]{{\lvert #2 \rvert}_{#1}}

\newcommand{\Biggnorm}[2][]{{\Biggl\lVert #2 \Biggr\rVert}_{#1}}

\newcommand{\quasinorm}[2][]{{[ #2 ]}_{#1}}

\newcommand{\Biggquasinorm}[2][]{{\Biggl[ #2 \Biggr]}_{#1}}

\newcommand{\st}{\;:\;}

\newcommand{\Rset}{\mathbb{R}}

\newcommand{\dif}{\,\mathrm{d}}

\newcommand{\eofs}{\,}

\newcommand{\R}{\ensuremath{\mathbb{R}}}
\renewcommand{\S}{\ensuremath{\mathbb{S}}}

\usepackage{enumerate}

\usepackage[T1]{fontenc}
\usepackage[utf8]{inputenc}
\usepackage{geometry}

\usepackage{lmodern} 
\usepackage{microtype}

\usepackage{amssymb}
\usepackage{amsmath}
\usepackage{amsthm}
\usepackage{esint}
\usepackage{mathtools}

\usepackage{mathrsfs}

\usepackage[nameinlink%
,capitalize%
]{cleveref}

\newtheorem{theorem}{Theorem}[section]
\newtheorem{proposition}[theorem]{Proposition}

\newtheorem{corollary}[theorem]{Corollary}

\theoremstyle{definition}

\theoremstyle{remark}
\newtheorem{remark}[theorem]{Remark}

\numberwithin{equation}{section}

\begin{document}

\begin{abstract}
We establish the equivalence between the Sobolev semi-norm $\|\nabla u\|_{L^p}$ and a quantity obtained when replacing strong $L^p$ by weak $L^p$ in the Gagliardo semi-norm $|u|_{W^{s,p}}$ computed at $s = 1$. As corollaries we derive alternative estimates in some exceptional cases (involving $W^{1,1}$) where the ``anticipated'' fractional Sobolev and Gagliardo-Nirenberg inequalities fail.
\end{abstract}

\maketitle
%\thispagestyle{firststyle}
%\ifthenelse{\boolean{shortarticle}}{\ifthenelse{\boolean{singlecolumn}}{\abscontentformatted}{\abscontent}}{}

% If your first paragraph (i.e. with the \dropcap) contains a list environment (quote, quotation, theorem, definition, enumerate, itemize...), the line after the list may have some extra indentation. If this is the case, add \parshape=0 to the end of the list environment.
%\dropcap{T}his PNAS journal template is provided to help you write your work in the correct journal format.  Instructions for use are provided below. 

%Note: please start your introduction without including the word ``Introduction'' as a section heading (except for math articles in the Physical Sciences section); this heading is implied in the first paragraphs. 

\section{Introduction}

Fractional Sobolev spaces \(W^{s, p}\) (also called Slobodeskii spaces) play a major role in many questions involving partial differential equations.
On $\R^N$, $N \geq 1$, they are associated with the Gagliardo semi-norm
\begin{equation}
\label{eq_Jai7ahg2tahz4ua3zaishatu}
\abs{u}_{W^{s, p}}^p 
\defeq \iint\limits_{\Rset^N \times \Rset^N} \frac{\abs{u (x) - u (y)}^p}{\abs{x - y}^{N + sp}} \dif x \dif y\eofs
\end{equation}
where \(0 < s < 1\) and \(1 \le p < \infty\).
A well-known ``drawback'' of the Gagliardo semi-norm is that one does not recover the Sobolev semi-norm $\|\nabla u\|_{L^p}^p$ if one takes $s = 1$ in \eqref{eq_Jai7ahg2tahz4ua3zaishatu}. In fact, for every \(1 \le p < \infty\) and every measurable function \(u\),
\begin{equation}
\begin{split}
\label{eq_xoh6aiwef2eikie7thahJ8Ee}
  \Biggnorm[L^p (\Rset^N \times \Rset^N)]
  {\frac{u (x) - u (y)}{\abs{x - y}^{\frac{N}{p} + 1}}}^p
  &= \iint\limits_{\Rset^N \times \Rset^N} \frac{\abs{u (x) - u (y)}^p}{\abs{x - y}^{N + p}} \dif x \dif y \\
  &= \infty
\end{split}
\end{equation}
unless \(u\) is a constant; see \cite{Bourgain_Brezis_Mironescu_2000} and also \cite{Brezis_2002}.

One way to recover \(\norm[L^p]{\nabla u}^p\) out of the Gagliardo semi-norms is to consider the quantity \((1 - s)\seminorm[W^{s, p}]{u}^p\) with \(0 < s <1\) and show that it converges as \(s \nearrow 1\) to a multiple of \(\norm[L^p]{\nabla u}^p\). This is a special case of the BBM formula in Bourgain-Brezis-Mironescu \cite{Bourgain_Brezis_Mironescu_2001} (see also \cite{Brezis_2002,VanSchaftingen_Willem_2004,Davila}), which furthermore enters (when $p=1$ and $u$ is a characteristic function) in the study of  ``nonlocal minimal surfaces'' and ``$s$-perimeters'' (see e.g. \cite{ADM,CV,FFMMM}).

The primary goal of this paper is to propose an alternative route to repair this ``defect'', simply replacing the $L^p$ norm $\|\cdot\|_{L^p}$ in \eqref{eq_xoh6aiwef2eikie7thahJ8Ee} by the Marcinkiewicz $M^p$ (i.e.\ weak~$L^p$) quasi-norm $[\,\cdot\,]_{M^p}$. The central result of the paper is \cref{theorem_Mp}.% described below.

In a seemingly different direction, it is well-known that in some exceptional cases the ``anticipated'' fractional Sobolev-type and Gagliardo--Nirenberg-type estimates may fail (in particular when they involve $\|\nabla u\|_{L^1}$). 
A second goal of this paper is to discuss a partial list of such failures (for a complete list see \cite{Brezis_Mironescu_2018,Brezis_Mironescu_2019}) and to present alternative (weaker) estimates where strong \(L^p\) is replaced by weak~\(L^p\); see \cref{theorem_W11_Wsp,theorem_intro,theorem_intro_2}. They can all be derived as immediate consequences of \cref{theorem_Mp} applied with $p = 1$. 

Our main result is the following

%\subsection{Fixing a \texorpdfstring{``defect''}{“defect”} of the Gagliardo semi-norm \texorpdfstring{$|\cdot|_{W^{s,p}}$}{|·|Ws,p} when \texorpdfstring{$s = 1$}{s = 1}}

\begin{theorem} \label{theorem_Mp}
For every $N \geq 1$, there exist constants $c = c(N) > 0$ and $C = C(N)$ such that 
\begin{equation} \label{eq:new_char_sob_norm}
c^p \|\nabla u\|_{L^p(\R^N)}^p \leq \Big [ \frac{u(x)-u(y)}{|x-y|^{\frac{N}{p}+1}} \Big]_{M^p(\R^N \times \R^N)}^p \leq C \|\nabla u\|_{L^p(\R^N)}^p
\end{equation}
for all $u \in C^{\infty}_c(\R^N)$ and all $1 \leq p < \infty$. 
\end{theorem}

Here \(M^p = L^p_w = L^{p, \infty}\), \(1 \le p <\infty\), is the Marcinkiewicz (=weak \(L^p\)) space modelled on \(L^p\), and
\[
\quasinorm[M^p (\R^N \times \R^N)]{f}^p \defeq \sup_{\lambda > 0} \lambda^p \mathcal{L}^{2 N} \bigl(\{ x\in \R^N \times \R^N \st \abs{f (x)} \ge \lambda\}\bigr)
\]
where \(\mathcal{L}^{2 N}\) is the Lebesgue measure on \(\R^N \times \R^N\).
%(see for example \cite[Chapter 5]{Castillo_Rafeiro_2016}, \cite[Section 1.1]{Grafakos_2014}).

In fact, one can sharpen substantially the lower bound in \eqref{eq:new_char_sob_norm}.

\begin{theorem} \label{theorem_reverse}
Let $N \geq 1$, $1 \leq p < \infty$ and $u \in C^{\infty}_c(\R^N)$. For $\lambda > 0$, let 
\begin{equation} \label{E_lambda}
E_{\lambda} := \biggl\{ (x,y) \in \R^{N} \times \R^{N} \colon x \ne y, \, \frac{|u(x)-u(y)|}{|x-y|^{\frac{N}{p} + 1}} \geq \lambda \biggr\}.
\end{equation}
Then
\begin{equation}
\lim_{\lambda \to \infty} \lambda^p \mathcal{L}^{2 N}(E_{\lambda}) = \frac{k(p,N)}{N}  \|\nabla u\|_{L^p(\R^N)}^p.
\end{equation}
\end{theorem}
Here
\begin{equation} \label{eq:kpN}
k(p,N) := \int_{\S^{N-1}} |e \cdot \omega|^p \dif \omega,
\end{equation}
$\mathbb{S}^{N-1}$ denotes the unit sphere in $\mathbb{R}^N$, and $e$ is any unit vector in $\R^N$.

Some comments concerning the above results are in order.

First, the validity of the upper bound in \eqref{eq:new_char_sob_norm} when \(p = 1\) is quite remarkable and somewhat unexpected. In fact, a natural strategy to establish this upper bound (such as the one presented in \cref{remark_bojarski} below) requires a strong type estimate for the maximal function (of the gradient of \(u\)), which holds when \(p > 1\), but notoriously fails at the end-point \(p = 1\). We overcome this difficulty by applying the Vitali covering lemma in a rather unconventional way which allows us to bypass the obstruction commonly arising at \(p = 1\) in this kind of situation. Thus, the hard core of the proof of the upper bound in \eqref{eq:new_char_sob_norm} concerns the case \(p = 1\). As it turns out, we can furthermore derive the case \(p > 1\) from the case \(p = 1\), at a crucial step of the argument. When \(p = 1\) and \(N = 1\) the upper bound in \eqref{eq:new_char_sob_norm} amounts to the following ``innocuous looking'' calculus inequality
\begin{equation}
\label{eq_aizohng7ahf7DaeNgaihaiTh}
\mathcal{L}^2 \left\{ (x,y) \in \R^2 \st |u(x)-u(y)| \geq |x-y|^2 \right\} \leq C \int_{\R} |u'(t)| dt
\end{equation}
for all \(u \in C^{\infty}_c(\R)\), where \(C\) is a universal constant. Surprisingly, this estimate seems to be new and our proof is more involved than expected!

Next, the lower bound in \eqref{eq:new_char_sob_norm} is a consequence of \cref{theorem_reverse}. The proof of \cref{theorem_reverse} involves new ideas, partially inspired from techniques developed in \cite{Bourgain_Brezis_Mironescu_2001}; actually, the constant \(k(p,N)\) in \eqref{eq:kpN} already appeared in the BBM formula \cite[Theorem 1.2]{Bourgain_Brezis_Mironescu_2001}. 

The proof of the upper bound in \eqref{eq:new_char_sob_norm} is presented in \cref{sect:newSobnorm}. The proof of \cref{theorem_reverse} is presented in \cref{sect:limitseminorm}.

The assertions in \cref{theorem_Mp,theorem_reverse}, which are stated for convenience when $u \in C^{\infty}_c(\R^N)$, suggest that similar conclusions hold under minimal regularity assumptions on \(u\), and that the Sobolev space \(\dot{W}^{1, p}\), \(1 < p < \infty\) (respectively $\dot{BV}$ when $p = 1$), can be identified with the space of measurable functions \(u\) satisfying \(\sup_{\lambda > 0}\lambda^p \mathcal{L}^{2N} (E_\lambda)< \infty\), or just \(\limsup_{\lambda \to \infty}\lambda^p \mathcal{L}^{2N} (E_\lambda) < \infty\).
One should also be able to replace $\R^N$ by domains $\Omega \subset \R^N$, etc. 
We will return to this circle of ideas in a forthcoming paper.

\section{Proof of \texorpdfstring{\cref{theorem_Mp}}{Theorem 1.1}}  \label{sect:newSobnorm}
  
As already mentioned the lower bound part is a consequence of \cref{theorem_reverse} whose proof is presented in Section 3: indeed, if $E_{\lambda}$ is as in \eqref{E_lambda}, then 
\[
\begin{split}
\Big [ \frac{u(x)-u(y)}{|x-y|^{\frac{N}{p}+1}} \Big]_{M^p(\R^N \times \R^N)}^p
= \sup_{\lambda > 0} \lambda^p \mathcal{L}^{2N}(E_{\lambda}) \geq \lim_{\lambda \to \infty} \lambda^p \mathcal{L}^{2N}(E_{\lambda})
\end{split}
\]
and H\"{o}lder's inequality gives
\[
\begin{split}
\left( \frac{k(p,N)}{N} \right)^{1/p} 
\geq \frac{k(1,N)}{\sigma_{N-1}^{1-\frac{1}{p}} N^{\frac{1}{p}} } = \frac{k(1,N)}{\sigma_{N-1}} \left( \frac{\sigma_{N-1}}{N} \right)^{1/p} 
\geq k(1,N) \min\left\{\frac{1}{N}, \frac{1}{\sigma_{N-1}} \right\} := c(N),
\end{split}
\]
where $\sigma_{N-1}$ denotes the surface area of $\mathbb{S}^{N-1}$.
Therefore we concentrate here on the upper bound. 

The key of our proof is the following proposition, which when $\gamma = 1$ and $f = u'$ gives \eqref{eq_aizohng7ahf7DaeNgaihaiTh}, and thus yields the desired upper bound for the $p=1$ case of \cref{theorem_Mp} in dimension $N = 1$.

\begin{proposition} \label{weighted1d}
There exists a universal constant $C$ such that for all $\gamma > 0$ and all $f \in C_c(\R)$, we have
\[
\iint_{E(f,\gamma)} |x-y|^{\gamma-1} \dif x \dif y \leq C \frac{5^{\gamma}}{\gamma} \|f\|_{L^1(\R)},
\]
where
\[
E(f,\gamma) := \Big\{(x,y) \in \R \times \R \colon x \ne y, \, \Big|\int_y^x f \Big| \geq |x-y|^{\gamma+1} \Big\}.
\]
\end{proposition}

\begin{proof}
Since $E(f,\gamma) \subseteq E(|f|,\gamma)$, without loss of generality assume $f$ is non-negative. Let $X$ be the collection of all non-trivial closed intervals $I \subset \R$ such that 
\begin{equation} \label{eq:Icond}
\int_I f \geq |I|^{\gamma+1}.
\end{equation}
(Here an interval is said to be non-trivial if it has positive length, and we used $|I|$ to denote the length of the interval.) Then 
\begin{equation}
E(f,\gamma) \subseteq \bigcup_{I \in X} I \times I.
\end{equation}
The lengths of all intervals in $X$ are bounded by $\|f\|_{L^1}^{1/(\gamma+1)} < \infty$. Hence we may apply the Vitali covering lemma, and choose a subcollection $Y$ of $X$, so that $Y$ consists of a family of pairwise disjoint intervals $J$ from $X$, and every $I \in X$ is contained in $5J$ for some $J \in Y$ (see e.g. \cite{EvansGariepy}, Claim in the proof of Theorem 1, Section 1.5). It follows that
\begin{equation} \label{eq:claim}
E(f,\gamma) \subseteq \bigcup_{J \in Y} (5J) \times (5J),
\end{equation}
where $5J$ is the interval with the same center as $J$ but 5 times the length.
As a result, we see that 
\begin{equation} \label{intENp}
\begin{split}
\iint_{E(f,\gamma)}\hspace{-1em} |x-y|^{\gamma-1} \dif x \dif y 
&\leq \sum_{J \in Y} \iint_{5J \times 5J} |x-y|^{\gamma-1} \dif x \dif y \\
&= \frac{10 \cdot 5^{\gamma}}{\gamma (\gamma + 1)} \sum_{J \in Y} |J|^{\gamma+1}.
\end{split}
\end{equation}
(Here we used $\gamma > 0$ to integrate in $x$ and $y$.)
But for each $J \in Y$, we have $J \in X$, so
\[
|J|^{\gamma+1} \leq \int_J f.
\]
Plugging this back into \eqref{intENp}, we obtain
\begin{equation}
\begin{split}
\iint_{E(f,\gamma)} |x-y|^{\gamma-1} \dif x \dif y 
&\leq C \frac{5^{\gamma}}{\gamma} \sum_{J \in Y} \int_J f\\
&\leq C \frac{5^{\gamma}}{\gamma} \|f\|_{L^1(\R)},
\end{split}
\end{equation}
the last inequality following from the disjointness of the different $J \in Y$. This completes the proof of the proposition.
\end{proof}

To prove the upper bound in \cref{theorem_Mp} when $N > 1$ or $p > 1$, Proposition~\ref{weighted1d} still proves to be useful. Via the method of rotation, it implies the following proposition:

\begin{proposition} \label{setNd}
For any positive integer $N$, there exists a constant $C=C(N)$ such that for all $F \in C_c(\R^N)$, we have
\[
\mathcal{L}^{2N}(E(F)) \leq C \|F\|_{L^1(\R^N)}
\]
where
\[
E(F) := \Big \{ (x,y) \in \R^N \times \R^N \colon x \ne y, \, \Big| \int_y^x F \Big| \geq |x-y|^{N+1} \Big\}.
\]
Here $\int_x^y F$ is the integral of $F$ along the line segment in $\R^N$ connecting $x$ to $y$.
\end{proposition}

\begin{proof}
Again without loss of generality, we may assume that $F$ is non-negative. By a change of variable, 
\[
\begin{split}
\mathcal{L}^{2N}(E(F)) 
&= \mathcal{L}^{2N} \Big( \Big\{ (x,y) \in \R^N \times \R^N \colon y \ne 0, \,
\int_0^{|y|} F \Big(x+t \frac{y}{|y|} \Big) \dif t \geq |y|^{N+1} \Big\} \Big) \\
&= \int_{\R^N} \mathcal{L}^{N} \Big( \Big \{y \in \R^N \setminus \{0\} \colon 
\int_0^{|y|} F\Big(x+t \frac{y}{|y|} \Big) \dif t \geq |y|^{N+1} \Big \} \Big) \dif x .
\end{split}
\]
Using polar coordinates to evaluate the integrand, we get
\[
\mathcal{L}^{2N}(E(F)) = \int_{\R^N} \int_{\S^{N-1}} \int_{E(F,x,\omega)} r^{N-1} \dif r \dif \omega \dif x ,
\]
where
\[
E(F,x,\omega) := \Big \{r \in (0,\infty) \colon \int_0^{r} F(x+t \omega) \dif t \geq r^{N+1} \Big\}.
\]
We now use Fubini to interchange the integral over $\R^N$ and $\S^{N-1}$. Then for each $\omega \in \S^{N-1}$, we foliate $\R^N$ as an orthogonal sum $\omega^{\perp} \oplus \R\omega$, where $\omega^{\perp}$ is the subspace of all $x \in \R^N$ that is orthogonal to $\omega$. Hence
\begin{equation} \label{eq:rotation}
\mathcal{L}^{2N}(E(F)) = \int_{\S^{N-1}} \int_{\omega^{\perp}} \int_{\R} \int_{E(F,x'+s\omega,\omega)} r^{N-1} \dif r \dif s \dif x ' \dif \omega.
\end{equation}
We now estimate the inner most double integral. For each fixed $\omega \in \S^{N-1}$ and each $x' \in \omega^{\perp}$, let $f_{\omega,x'} \in C^{\infty}_c(\R)$ be a function of one variable defined by
\[
f_{\omega,x'}(t) = F(x'+t\omega), \quad t \in \R.
\]
Then
\[
E(F,x'+s\omega,\omega)= \Big\{r \in (0,\infty) \colon \int_0^r f_{\omega,x'}(s+t) \dif t \geq r^{N+1} \Big\},
\]
so change of variables again gives
\begin{equation} \label{eq:EFreduced}
 \int_{\R} \int_{E(F,x'+s\omega,\omega)} r^{N-1} \dif r \dif s
= \frac{1}{2} \iint_{E(f_{\omega,x'},N)} |r-s|^{N-1} \dif r \dif s
\end{equation}
where
\begin{equation*}
E(f_{\omega,x'},N)
\defeq  \Big\{ (s,r) \in \R \times \R \colon s \ne r, \, \Big| \int_s^r f_{\omega,x'} \Big| \geq |r-s|^{N+1} \Big\} 
\end{equation*}
as in Proposition~\ref{weighted1d} (the factor $1/2$ accounts for the fact that in the integral on the left hand side of \eqref{eq:EFreduced} we are only working with those $(s,r) \in E(f_{\omega,x'},N)$ with $s < r$). Appealing to Proposition~\ref{weighted1d} with $\gamma = N$, we may now estimate the double integral in the $(r,s)$ variables on the right hand side of \eqref{eq:rotation}. We obtain
\[
\begin{split}
\mathcal{L}^{2N}(E(F)) 
&\leq \frac{C}{2} \cdot \frac{5^N}{N} \int_{\S^{N-1}} \int_{\omega^{\perp}} \int_{\R} f_{\omega,x'}(t) \dif t \dif x ' \dif \omega \\
&= \frac{C}{2} \cdot \frac{5^N  \sigma_{N-1}}{N} \|F\|_{L^1(\R^N)},
\end{split}
\]
the last equality holding because for every $\omega \in \S^{N-1}$,
\[
\int_{\omega^{\perp}} \int_{\R} f_{\omega,x'}(t) \dif t \dif x ' = \int_{\omega^{\perp}} \int_{\R} F(x' + t\omega) \dif t \dif x ' = \|F\|_{L^1(\R^N)}.\qedhere
\]
\end{proof}

The upper bound in \cref{theorem_Mp} follows easily from Proposition~\ref{setNd}. 

\begin{proof}[Proof of \cref{theorem_Mp}, the upper bound]
Since $u \in C^{\infty}_c(\R^N)$, by H\"{o}lder's inequality, for every $1 \leq p < \infty$,
\[
|u(x)-u(y)| \leq \Big| \int_y^x |\nabla u| \Big| \leq |x-y|^{1-\frac{1}{p}} \Big| \int_y^x |\nabla u|^p \Big|^{\frac{1}{p}},
\] 
so for $\lambda > 0$,
\begin{multline}
\label{eq_hainoubaighohB7oonoo3cio}
 \Big\{ (x,y) \in \R^N \times \R^N \colon x \ne y, \, \frac{|u(x)-u(y)|}{|x-y|^{\frac{N}{p}+1}} \geq \lambda \Big\} \\
 \subseteq  \Big\{ (x,y) \in \R^N \times \R^N \colon x \ne y, \, 
 \Big| \int_y^x \frac{|\nabla u|^p}{\lambda^p} \Big| \geq |x-y|^{N+1} \Big\}.
\end{multline}
Applying Proposition~\ref{setNd} to $F := |\nabla u|^p / \lambda^p$, we see that 
\begin{multline*}
 \mathcal{L}^{2N} \Big( \Big\{ (x,y) \in \R^N \times \R^N \colon x \ne y, \, \frac{|u(x)-u(y)|}{|x-y|^{\frac{N}{p}+1}}
\geq \lambda \Big\} \Big) 
\leq  \frac{C}{\lambda^p} \|\nabla u\|_{L^p(\R^N)}^p
\end{multline*}
with $C = C(N)$, as desired. 
\end{proof}
  
\begin{remark}
\label{remark_bojarski}
When \(p > 1\), the upper bound has a short proof relying on an estimate of the difference quotient by the maximal function of the gradient.
The main ingredient is the following so-called Lusin-Lipschitz inequality,
\begin{equation}
  \label{eq_laigieShuyaJiexua0Xahng7}
  \abs{u (x) - u (y)} \le C \abs{x - y} \bigl(\mathcal{M} \abs{\nabla u} (x) + \mathcal{M} \abs{\nabla u} (y) \bigr)\eofs,
\end{equation}
where \(\mathcal{M} f\) denotes the Hardy--Littlewood maximal function of \(f\); see  \cite[p. 404]{Hajlasz_1996} for a complete proof, and \cite{BCD,ABT} for recent developments.
Inequality \eqref{eq_laigieShuyaJiexua0Xahng7} implies that 
\begin{multline}
\label{eq_eiMu0oci5aeR5Owaeregheet}
\biggl\{ (x, y) \in \Rset^N \times \Rset^N \st x \ne y, \, \frac{\abs{u (x) - u (y)}}{\abs{x - y}^{\frac{N}{p} + 1}} \ge \lambda\biggr\}\\
   \shoveleft{\subseteq \Bigl\{(x, y) \in \Rset^N \times \Rset^N \st} 
  \shoveright{\abs{x - y}^\frac{N}{p} \le C \lambda^{-1} \bigl(\mathcal{M} \abs{\nabla u} (x) + \mathcal{M} \abs{\nabla u} (y)\bigr) \Bigr\}}\\
  \subseteq  
  \Bigl\{(x, y) \in \Rset^N \times \Rset^N \st \abs{x - y}^\frac{N}{p} \le 2 C \lambda^{-1} \mathcal{M} \abs{\nabla u} (x) \Bigr\}\\
  \cup \Bigl\{(x, y) \in \Rset^N \times \Rset^N \st \abs{x - y}^\frac{N}{p} \le 2 C \lambda^{-1} \mathcal{M} \abs{\nabla u} (y) \Bigr\}\eofs.
\end{multline}
and thus that 
\begin{multline*}
 \lambda^p 
\mathcal{L}^{2N} \biggl(\biggl\{ (x, y) \in \Rset^N \times \Rset^N \st \frac{\abs{u (x) - u (y)}}{\abs{x - y}^{\frac{N}{p} + 1} }\ge \lambda\biggr\}\biggr) 
  \le  C'(p,N) \int_{\Rset^N} \bigl(\mathcal{M} \abs{\nabla u}\bigr)^p\eofs.
\end{multline*}
For $1 < p < \infty$, the maximal function theorem then implies
\[
\Big [ \frac{u(x)-u(y)}{|x-y|^{\frac{N}{p}+1}} \Big]_{M^p(\R^N \times \R^N)} \leq C(p,N) \|\nabla u\|_{L^p(\R^N)}.
\]
The constant coming the maximal function theorem deteriorates as \(p \searrow 1\).
\end{remark}

\section{Proof of \texorpdfstring{\cref{theorem_reverse}}{Theorem 1.2}} \label{sect:limitseminorm}
 
 We now prove \cref{theorem_reverse} and hence the lower bound in \cref{theorem_Mp}.
 
We will use the inequalities 
\begin{equation} \label{eq:linear}
|u(x) - u(y)| \leq L |x-y| \quad \forall x , y \in \R^N
\end{equation}
with $L := \|\nabla u\|_{L^{\infty}(\R^N)}$ and 
\begin{equation} \label{eq:quadratic}
|u(x) - u(y) - \nabla u(x) \cdot (x-y)| \leq A |x-y|^2 \quad \forall x , y \in \R^N
\end{equation}
with $A := \|\nabla^2 u\|_{L^{\infty}(\R^N)}$. 

Fix $x \in \R^N$ and a direction $\omega \in \S^{N-1}$. For a large positive number $\lambda$, consider the set $E_{\lambda}(x,\omega)$ consisting of all $y \in \R^N$ such that $y-x$ is a positive multiple of $\omega$ and $(x,y) \in E_{\lambda}$. 
We will determine two numbers $\underline{R} = \underline{R}(x,\omega,\lambda)$ and $\overline{R} = \overline{R}(x,\omega,\lambda)$ such that 
\[
\{x + r \omega  \colon r \in (0, \underline{R}]\} \subseteq E_{\lambda}(x,\omega) \subseteq \{x + r \omega \colon r \in (0, \overline{R}]\}.
\]
Using polar coordinates, we then deduce that 
\begin{equation} \label{eq:Exmeas}
\begin{split}
\frac{1}{N} \int_{\S^{N-1}} \hspace{-.8em} \underline{R}(x,\omega,\lambda)^N \dif \omega &\leq \mathcal{L}^N \big( \{y \in \R^N \colon (x,y) \in E_{\lambda} \} \big)\\ 
&\leq \frac{1}{N} \int_{\S^{N-1}} \overline{R}(x,\omega,\lambda)^N \dif \omega.
\end{split}
\end{equation}

From \eqref{eq:quadratic} we have
\[
|u(x)-u(y)| \geq |\nabla u(x) \cdot (x-y)| - A |x-y|^2 \geq \lambda |x-y|^{1+\frac{N}{p}}
\]
provided 
\begin{equation} \label{eq:lowercond}
A r + \lambda r^{N/p} \leq |\nabla u(x) \cdot \omega|
\end{equation}
where $r := |y-x|$ and $\omega = \frac{y-x}{|y-x|} \in \S^{N-1}$. 

Fix $\delta > 0$ arbitrarily small. Then by \eqref{eq:lowercond}, the conditions 
\[
A r \leq \delta |\nabla u (x) \cdot \omega| 
\quad
\text{and} 
\quad
\lambda r^{N/p} \leq (1-\delta) |\nabla u(x) \cdot \omega| 
\]
imply that $(x,y) \in E_{\lambda}$. Thus we can take $\underline{R}$ to be defined by 
\[
\underline{R}(x,\omega,\lambda)^N := \min \Big\{ \frac{\delta^N}{A^N} |\nabla u(x) \cdot \omega|^N, \frac{(1-\delta)^p}{\lambda^p} |\nabla u(x) \cdot \omega|^p \Big\}.
\]
From \eqref{eq:Exmeas} we have
\begin{multline*}
\lambda^p \mathcal{L}^{2N}(E_{\lambda})
\geq \frac{1}{N} \iint   \min \Big\{ \frac{\lambda^p \delta^N}{A^N} |\nabla u(x) \cdot \omega|^N, (1-\delta)^p |\nabla u(x) \cdot \omega|^p \Big\} \dif \omega \dif x,
\end{multline*}
where the integral is over all points \((x,\omega) \in \R^N \times \S^{N-1}\) with \(\nabla u(x) \cdot \omega \ne 0\),
and by monotone convergence,
\[
\liminf_{\lambda \to \infty} \lambda^p \mathcal{L}^{2N}(E_{\lambda}) \geq \frac{(1-\delta)^p}{N} \int_{\R^N} \int_{\S^{N-1}} |\nabla u(x) \cdot \omega|^p \dif \omega \dif x.
\]
Since $\delta > 0$ is arbitrary, we conclude that
\[
\liminf_{\lambda \to \infty} \lambda^p \mathcal{L}^{2N}(E_{\lambda}) \geq \frac{k(p,N)}{N}  \int_{\R^N} |\nabla u(x)|^p \dif x
\]
where $k(p,N)$ is defined by \eqref{eq:kpN}.

It remains to establish that
\begin{equation} \label{eq:limsupE}
\limsup_{\lambda \to \infty} \lambda^p \mathcal{L}^{2N}(E_{\lambda}) \leq \frac{k(p,N)}{N}  \int_{\R^N} |\nabla u(x)|^p \dif x.
\end{equation}
From \eqref{eq:quadratic} we have
\[
|u(x) - u(y)| \leq |\nabla u(x) \cdot (x-y)| + A |x-y|^2
\]
and thus if $(x,y) \in E_{\lambda}$ we obtain
\begin{equation} \label{eq:A}
\lambda r^{N/p} \leq |\nabla u(x) \cdot \omega| + A r
\end{equation}
where again $r = |y-x|$ and $\omega = \frac{y-x}{|y-x|} \in \S^{N-1}$. On the other hand, if $(x,y) \in E_{\lambda}$, we have from \eqref{eq:linear} that
\begin{equation} \label{eq:L}
\lambda r^{N/p} \leq L.
\end{equation}
Inserting \eqref{eq:L} into \eqref{eq:A} yields 
\begin{equation} \label{eq:rhoupper}
\lambda r^{N/p} \leq |\nabla u(x) \cdot \omega| + A \Big( \frac{L}{\lambda} \Big)^{p/N}.
\end{equation}
In what follows we will consider only 
\begin{equation} \label{eq:lambdalarge}
\lambda > L.
\end{equation} 
Observe that if $\text{dist}(x, \text{supp}\, u) > 1$ then
\begin{equation} \label{eq:Exempty}
\{y \in \R^N \colon (x,y) \in E_{\lambda} \} = \emptyset.
\end{equation}
Indeed by \eqref{eq:L} and \eqref{eq:lambdalarge} we have, for any $(x,y) \in E_{\lambda}$, that $|x-y| \leq 1$. So if $\text{dist}(x, \text{supp}\, u) > 1$ and $y \in \R^N$ is such that $(x,y) \in E_{\lambda}$, then $y \notin \text{supp}\, u$, from which it follows that $\lambda |x-y|^{\frac{N}{p} + 1} \leq |u(x)-u(y)| = 0$, i.e. $x = y$, which is a contradiction since $(x,x) \notin E_{\lambda}$. 

Using \eqref{eq:rhoupper} and \eqref{eq:Exempty} we may take $\overline{R}$ to be
\begin{equation*}
\overline{R}(x,\omega,\lambda)^N
= 
\begin{cases}
\Big( \frac{\displaystyle |\nabla u(x) \cdot \omega| + A \big( \tfrac{L}{\lambda} \big)^{p/N}}{\lambda} \Big)^p &\quad \text{if $\text{dist}(x,\text{supp}\, u) \leq 1$} \\
0 &\quad \text{otherwise}.
\end{cases}
\end{equation*}
Consequently from \eqref{eq:Exmeas}
\begin{multline}
\lambda^p \mathcal{L}^{2N}(E_{\lambda}) 
\leq  \frac{1}{N} \int_{\R^N} \int_{\S^{N-1}} \mathbf{1}_{\text{dist}(x,\text{supp}\, u) \leq 1}  \Big( |\nabla u(x) \cdot \omega| + A \big( \tfrac{L}{\lambda} \big)^{p/N} \Big)^p \dif \omega \dif x
\end{multline}
which yields \eqref{eq:limsupE} by dominated convergence. \qed

\section{Fixing a ``defect'' of a fractional Sobolev-type estimate}

A typical fractional Sobolev-type estimate would assert that 
\begin{equation*}
 \dot{W}^{1, 1} (\Rset^N) \subset W^{s, p} (\Rset^N), \qquad \text{with continuous injection},
\end{equation*}
for every \(N \ge 1\) and every \(0 < s < 1\), where \(1 < p <\infty\) is defined by 
\begin{equation}
 \frac{1}{p} = 1 - \frac{1 - s}{N}\eofs.
\end{equation}
This amounts to 
\begin{equation}
  \label{eq_shaiphoBiephoum3ez2uizee}
  \Biggnorm[L^p (\Rset^N \times \Rset^N)]
  {\frac{u (x) - u (y)}{\abs{x - y}^{\frac{N}{p} + s}}}
  \hspace{-1em}
  \le C \norm[L^1  (\Rset^N)] {\nabla u}\eofs, \quad 
  \forall u \in C^\infty_c (\Rset^N)\eofs.
\end{equation}
It turns out that \eqref{eq_shaiphoBiephoum3ez2uizee} holds when \(N \ge 2\) but fails when \(N = 1\). (Estimate \eqref{eq_shaiphoBiephoum3ez2uizee} when \(N \ge 2\) is due to Solonnikov \cite{Solonnikov_1972}; see also \cite[Appendix D]{Bourgain_Brezis_Mironescu_2004} for a proof when \(N = 2\) which can be adapted to  any \(N \ge 2\) and \cite[Corollary 8.2]{VanSchaftingen_2013} for a proof based on cancellation properties of gradients in endpoint estimates \cite{Bourgain_Brezis_2007}.)
When \(N = 1\), \eqref{eq_shaiphoBiephoum3ez2uizee} reads as 
\begin{equation}
\label{eq_ro4Nahngohpeajahje6Aagh7}
  \Biggnorm[L^p (\Rset \times \Rset)]
  {\frac{u (x) - u (y)}{\abs{x - y}^{\frac{2}{p}}}}
  \le C \norm[L^1  (\Rset)] {u'}\eofs\eofs, \quad 
  \forall u \in C^\infty_c (\Rset)\eofs,
\end{equation}
which clearly fails for any \(p \in [1, \infty)\).
Indeed, take \(u = u_n\), a sequence of smooth functions converging to the characteristic function \(\mathbf{1}_I\) of a bounded interval \(I \subset \Rset\); note that the right-hand side of \eqref{eq_ro4Nahngohpeajahje6Aagh7} remains bounded while its left-hand side tends to infinity.
When \(p = 1\), the failure of \eqref{eq_ro4Nahngohpeajahje6Aagh7} is even more dramatic: the left-hand side is infinite for any measurable function \(u\) unless \(u\) is a constant, as mentioned in \eqref{eq_xoh6aiwef2eikie7thahJ8Ee}.%, see Bourgain--Brezis--Mironescu \cite{Bourgain_Brezis_Mironescu_2000} (see also \cite{Brezis_2002,DeMarco_Mariconda_Solimini_2008,RanjbarMotlagh}).

One way to repair the defect in \eqref{eq_shaiphoBiephoum3ez2uizee} when $N = 1$ consists of using again weak~$L^p$ instead of strong~$L^p$.

\begin{corollary}
\label{theorem_W11_Wsp}
There exists a constant \(C\) such that for every \(1 < p < \infty\), 
\begin{equation}
  \label{eq_Queewohz0oozoghie9rohjai}
    \Biggquasinorm[M^p (\Rset \times \Rset)]{\frac{u (x) - u (y)}{\abs{x - y}^{\frac{2}{p}}}} 
    \le 
    C
    \,
    \norm[L^1 (\Rset)]{u'}\eofs,\quad \forall u \in C^\infty_c (\Rset)\eofs.
\end{equation}
\end{corollary}

\begin{remark}
When $p = 2$ estimate \eqref{eq_Queewohz0oozoghie9rohjai} is originally due to Greco and Schiattarella \cite{Greco_Schiattarella}.
\end{remark}

%\cref{theorem_W11_Wsp} is an obvious consequence of \cref{theorem_intro} below (applied with $N = 1$) since $\|u\|_{L^{\infty}(\R)} \leq \|u'\|_{L^1(\R)}$. 

%\begin{remark}
%When $p = 2$ estimate \eqref{eq_Queewohz0oozoghie9rohjai} is originally due to Greco and Schiattarella \cite{Greco_Schiattarella}. The conclusion of \cref{theorem_W11_Wsp} is also valid when $p = 1$; this corresponds to the upper bound in \cref{theorem_Mp} with $N= 1$ and $p = 1$.
%\end{remark}

\section{Fixing a ``defect'' of some fractional Gagliardo-Nirenberg-type estimates}

We first consider a Gagliardo--Nirenberg-type inequality involving \(\dot{W}^{1, 1} (\Rset^N)\) and \(L^{p_1} (\Rset^N)\) with \(N \ge 1\) and \(1 \le p_1 \le \infty\).

Let \(\theta \in (0, 1)\) and set 
\begin{equation}
\label{eq_Am1iupeeku1apheiyie7Yuqu}
\begin{split}
  s &= \theta \cdot 0 + (1 - \theta) \cdot 1 = 1 - \theta 
\\
  \frac{1}{p} & = \frac{\theta}{p_1} + \frac{1 - \theta}{1}
  = \frac{\theta}{p_1} + (1 - \theta)\eofs.
\end{split}
\end{equation}
It is known that the estimate 
\begin{equation}
\begin{split}
  \label{eq_xieLichahZah4eetoof9lei2}
  \seminorm[W^{s, p} (\Rset^N)]{u}
  &=
  \Biggnorm[L^p (\Rset^N \times \Rset^N)]{\frac{u (x) - u (y)}{\abs{x - y}^{\frac{N}{p} + s}}}\\
  &
  \le 
  C \norm[L^{p_1} (\Rset^N)]{u}^\theta \norm[L^1 (\Rset^N)]{\nabla u}^{1 - \theta}
  \end{split}
\end{equation}
\begin{itemize}
\item 
\emph{holds} for every \(\theta \in (0, 1)\) when \(1 \le p_1 < \infty\); and
\item 
\emph{fails} for every \(\theta \in (0, 1)\) when \(p_1 = \infty\), 
\end{itemize}\parshape=0
see e.g.\ Brezis--Mironescu \cite{Brezis_Mironescu_2018} and the references therein.

We investigate here what happens when \(p_1 = \infty\) and the ``anticipated'' inequality 
\begin{equation}
%   \label{eq_shaiphoBiephoum3ez2uizee}
  \Biggnorm[L^p (\Rset^N \times \Rset^N)]
  {\frac{u (x) - u (y)}{\abs{x - y}^\frac{N + 1}{p}}}
  \le C \norm[L^\infty  (\Rset^N)] {u}^{1 - 1/p}
  \norm[L^{1} (\Rset^N)]{\nabla u}^{1/p}\eofs,
\end{equation}
for \(u \in C^\infty_c (\Rset^N)\), fails for every \(1 \le p < \infty\). (The argument is the same as above for the failure of \eqref{eq_ro4Nahngohpeajahje6Aagh7}.)

Our main result in this direction is 
\begin{corollary}
  \label{theorem_intro}
For every $N \geq 1$, there exists a constant \(C= C(N)\) such that for all \(1 < p <\infty\) and all \(u \in C^\infty_c (\Rset^N)\),
\begin{equation} \label{eq:Cor1.5}
\Biggquasinorm[M^p (\Rset^N \times \Rset^N)]
{\frac{u (x) - u (y)}{\abs{x - y}^\frac{N + 1}{p}}}\le C \norm[L^\infty  (\Rset^N)]{u}^{1 - 1/p}
\norm[L^1 (\Rset^N)]{\nabla u}^{1/p}\eofs.
\end{equation}
\end{corollary}

Finally, we turn to another situation, also involving \(\dot{W}^{1, 1}\), where the Gagliardo--Nirenberg-type inequality fails. 
Let \(0 < s_1 < 1\), \(1 < p_1 < \infty\) and \(0 < \theta < 1\).
Set
\begin{align}
  \label{eq_Ahqueep9eci9aikio5eequai}
  s &= \theta s_1 + (1 - \theta) &
  &\text{ and }&
  \frac{1}{p} & = \frac{\theta}{p_1} + (1 - \theta)\eofs.
\end{align}
It is known that  the estimate 
\begin{equation}
  \label{eq_chuofah2shaiZiemohqu5eej}
  \begin{split}
  \seminorm[W^{s, p} (\Rset^N)]{u}
  =
  \Biggnorm[L^p (\Rset^N \times \Rset^N)]{\frac{u (x) - u (y)}{\abs{x - y}^{\frac{N}{p} + s}}}
  \le 
  C \seminorm[W^{s_1, p_1} (\Rset^N)]{u}^\theta \norm[L^1 (\Rset^N)]{\nabla u}^{1 - \theta}
  \end{split}
\end{equation}
\begin{itemize}
\item 
\emph{holds} for every \(\theta \in (0, 1)\) when \(s_1 p_1 < 1\) (Cohen, Dahmen, Daubechies and DeVore \cite{Cohen_Dahmen_Daubechies_DeVore_2003}), and 
\item 
\emph{fails} for every \(\theta \in (0, 1)\) when \(s_1 p_1 \ge 1\) (Brezis and Mironescu \cite{Brezis_Mironescu_2018}). 
\end{itemize}
We investigate here what happens in the regime \(s_1 p_1 \ge 1\).
Our main result in this direction is

\begin{corollary}
  \label{theorem_intro_2}
For every $N \geq 1$, there exists a constant \(C = C(N)\) such that for any \(s_1 \in (0, 1)\), \(p_1 \in (1, \infty)\) with \(s_1 p_1 \ge 1\) and for any \(\theta \in (0, 1)\), we have 
  \begin{multline} \label{GN_Wsp}
    \Biggquasinorm[M^p (\Rset^N \times \Rset^N)]{\frac{u (x) - u (y)}{\abs{x - y}^{\frac{N}{p} + s}}}
    \le  C \seminorm[W^{s_1, p_1} (\Rset^N)]{u}^\theta \norm[L^1 (\Rset^N)]{\nabla u}^{1 - \theta}\eofs,
    \qquad \forall u \in C^\infty_c (\Rset^N) \eofs 
  \end{multline}  
where \(0 < s <1\) and \(1 < p<\infty\) are defined by \eqref{eq_Ahqueep9eci9aikio5eequai}.
\end{corollary}

The proofs of \cref{theorem_W11_Wsp,theorem_intro,theorem_intro_2} rely on our main \cref{theorem_Mp} and the details will appear in a forthcoming article.  

\bigskip

{\bf Acknowledgements.} This work was completed during two visits of J. Van Schaftingen to Rutgers University. He thanks H. Brezis for the invitation and  the Department of Mathematics for its hospitality.
P-L. Yung was partially supported by the General Research Fund CUHK14313716 from the Hong Kong Research Grant Council.
H. Brezis is grateful to C. Sbordone who communicated to him the interesting paper \cite{Greco_Schiattarella} by Greco and Schiattarella
which triggered our work.
We are indebted 
%to J. Serra for calling our attention to the papers \cite{Figalli_Serra,Gui_Li}, and 
to E. Tadmor for useful comments.

%\showacknow{} % Display the acknowledgments section

% Bibliography
\bibliographystyle{plain}
\bibliography{PNAS-Brezis-VanSchaftingen-Yung}

\end{document}